\documentclass[11pt]{article}

\usepackage{amsmath, amsthm, amsfonts, amssymb, mathrsfs, graphicx, paralist}
\usepackage[usenames, dvipsnames]{color}
\usepackage[margin=1in]{geometry} 
\usepackage[colorlinks=true, urlcolor=blue, linkcolor=blue, citecolor=blue]{hyperref}
\usepackage{tabmac}
\usepackage{fancyvrb}

\usepackage{tocloft}
\numberwithin{equation}{section}
\newtheorem{theorem}[equation]{Theorem}

\newtheorem{proposition}[equation]{Proposition}
\newtheorem{lemma}[equation]{Lemma}
\newtheorem{corollary}[equation]{Corollary}
\newtheorem{conjecture}[equation]{Conjecture}

\theoremstyle{definition}
\newtheorem{rmk}[equation]{Remark}
\newenvironment{remark}[1][]{\begin{rmk}[#1] \pushQED{\qed}}{\popQED \end{rmk}}
\newtheorem{eg}[equation]{Example}
\newenvironment{example}[1][]{\begin{eg}[#1] \pushQED{\qed}}{\popQED \end{eg}}
\newtheorem{defn}[equation]{Definition}

\newcommand{\bC}{\mathbf{C}}

\newcommand{\bF}{\mathbf{F}}

\newcommand{\bH}{\mathbf{H}}

\newcommand{\rH}{\mathrm{H}}

\newcommand{\cO}{\mathcal{O}}

\newcommand{\bP}{\mathbf{P}}

\newcommand{\bQ}{\mathbf{Q}}

\newcommand{\cS}{\mathcal{S}}

\newcommand{\bZ}{\mathbf{Z}}

\newcommand{\fh}{\mathfrak{h}}

\renewcommand{\phi}{\varphi}
\renewcommand{\emptyset}{\varnothing}

\newcommand{\ol}[1]{\overline{#1}}
\newcommand{\ul}[1]{\underline{#1}}

\newcommand{\arxiv}[1]{\href{http://arxiv.org/abs/#1}{{\tt arXiv:#1}}}

\makeatletter
\def\Ddots{\mathinner{\mkern1mu\raise\p@
\vbox{\kern7\p@\hbox{.}}\mkern2mu
\raise4\p@\hbox{.}\mkern2mu\raise7\p@\hbox{.}\mkern1mu}}
\makeatother

\renewcommand{\hom}{\operatorname{Hom}}

\DeclareMathOperator{\rank}{rank}

\DeclareMathOperator{\Sym}{Sym}

\DeclareMathOperator{\depth}{depth}

\DeclareMathOperator{\pdim}{pdim}

\newcommand{\GL}{\mathbf{GL}}
\newcommand{\SL}{\mathbf{SL}}

\DeclareMathOperator{\ch}{char}

\title{Representations of rational Cherednik algebras of $G(m,r,n)$ in positive characteristic}

\author{Sheela Devadas \and Steven V Sam}

\date{October 13, 2013}

\begin{document}

\maketitle

\begin{abstract}
We study lowest-weight irreducible representations of rational Cherednik algebras attached to the complex reflection groups $G(m, r, n)$ in characteristic $p$. Our approach is mostly from the perspective of commutative algebra. By studying the kernel of the contravariant bilinear form on Verma modules, we obtain formulas for Hilbert series of irreducible representations in a number of cases, and present conjectures in other cases. We observe that the form of the Hilbert series of the irreducible representations and the generators of the kernel tend to be determined by the value of $n$ modulo $p$, and are related to special classes of subspace arrangements. Perhaps the most novel (conjectural) discovery from the commutative algebra perspective is that the generators of the kernel can be given the structure of a ``matrix regular sequence'' in some instances, which we prove in some small cases.
\end{abstract}

\setcounter{tocdepth}{1}
\tableofcontents

\section{Introduction} 

In this paper, we undertake a study of the modular representation theory of the Cherednik algebras associated to the complex reflection groups $G(m,r,n)$, which are generalizations of groups of permutation matrices. We refer the reader to \cite{etingofma} for an introduction to these algebras over $\bC$.

Lowest-weight representations of Cherednik algebras have been studied in both characteristic $0$ and positive characteristic. In characteristic $0$, deep tools have been developed and connections between other aspects of representation theory and algebraic geometry have been found. However, the case of positive characteristic has been studied less, because of a lack of general tools. The representation theory of complex reflection groups becomes more complicated in positive characteristic, which makes the representation theory of the associated Cherednik algebras more interesting.

Via a Verma-like construction, the lowest-weight representations can be expressed as quotients of a free module over a polynomial ring. More precisely, one starts with an irreducible representation $\tau$ of the reflection group $G$ and the rank of the free module is $\dim \tau$. Our approach to the study of these representations is mostly from the perspective of commutative algebra. In the case of $\dim \tau = 1$, then the representation has a ring structure. The lowest-weight representations are always finite-dimensional in positive characteristic, so is always supported at the origin and there is no obvious geometry at one's disposal. We prove in some cases, and conjecture in other cases, that this ring can always be presented as a complete intersection inside of the coordinate ring of a subspace arrangement (which will be a union of flats for the corresponding reflection group). 

There are natural surjections $G(m,r,n) \to G(r,r,n)$ so any representation $\tau$ of $G(r,r,n)$ can be considered as a representation of $G(m,r,n)$. We show in \S\ref{sec:degeneration} that the calculation of the lowest-weight module for the Cherednik algebra of $G(m,r,n)$ reduces to the case of $m=r$. This includes the case when $\tau$ is trivial, which was discussed above. For the most part, our study is focused on this case. So the calculation of the lowest-weight representation amounts to finding the generators of a certain ideal $J$. By what we have just said, it suffices to understand the case $G(m,m,n)$. The behavior of the cases $m=1$ ($G(1,1,n) = \Sigma_n$ is the symmetric group) and $m>1$ are different. In some sense, the case $m>1$ is more combinatorial and easier to study. In these cases, we produce an ideal contained in the desired ideal $J$, and we conjecture that it is the full ideal. We can prove that this is correct in certain situations.

We mention some previous related work on Cherednik algebras. In the rank $1$ case, i.e., for the cyclic groups $\bZ/\ell$, the representation theory was studied by Latour \cite{latour}. Balagovi\'c and Chen studied the Cherednik algebras for $\GL_n(\bF_q)$ and $\SL_n(\bF_q)$ in \cite{balagovic, GL2Fp}. The non-modular case for the symmetric group $\Sigma_n$ (where the characteristic does not divide the order of the group) was studied by Bezrukavnikov, Finkelberg, and Ginzburg in the context of algebraic geometry in \cite{bfg}. This case was also studied by Gordon \cite{gordon}. Lian \cite{lian} studied some of the same groups in our paper, but for different parameters than the ones that we use for the Cherednik algebra.  

Now we summarize the contents of the paper. \S\ref{sec:background} consists of background material on reflection groups, Cherednik algebras, and commutative algebra. \S\ref{sec:subspace} contains some information on subspace arrangements which are relevant to our work, as discussed above. In \S\ref{sec:Gm1n}, we extend Gordon's work \cite{gordon} to the wreath products $G(m,1,n) = \Sigma_n \ltimes (\bZ/m)^n$. This is straightforward, but we restate the relevant proofs for completeness. In \S\ref{sec:degeneration} we relate representations of $G(m,r,n)$ and $G(r,r,n)$ as discussed above. For the case $\tau$ trivial, this reduces our study to $G(m,m,n)$, and this is discussed in \S\ref{sec:tau-triv}. In \S\ref{sec:dihedral} we give a complete analysis of the groups $G(m,m,2)$, which are the dihedral groups. In the last section \S\ref{sec:rank3}, we study the groups $G(m,m,3)$ and give some partial results. We see a simple example of a ``matrix regular sequence'' here and this is elaborated on in Remark~\ref{rmk:matrixkoszul}.

\subsection*{Acknowledgements}

This project was made possible by the PRIMES program organized by Pavel Etingof and Slava Gerovitch. We thank Pavel Etingof for suggesting this problem. We thank Martina Balagovi\'c, Pavel Etingof, and Carl Lian for useful discussions. The computer algebra systems SAGE \cite{sage}, GAP \cite{gap}, and Macaulay2 \cite{m2} were very helpful for obtaining our results.

Steven Sam was supported by an NDSEG fellowship and a Miller research fellowship.

\section{Background} \label{sec:background}

\subsection{Reflection groups}

We say that a matrix $s$ is a {\bf reflection} if $s$ has finite order, i.e., $s^N = 1$ for some $N > 0$, and if $\rank(1-s) = 1$. Over a field of characteristic 0, this implies that $s$ is diagonalizable, and $s$ is what is classically known as a complex reflection, or pseudo-reflection. However, we prefer to use just the terminology ``reflection''. Over a field of positive characteristic, we have allowed the possibility that $s$ is unipotent, e.g., in characteristic 2, the matrix $\left( \begin{smallmatrix} 1 & 1 \\ 0 & 1 \end{smallmatrix} \right)$ is a reflection. In this article, all of the reflections that appear in our examples will be diagonalizable. We only allow the possibility of unipotent reflections to be consistent with \cite{balagovic, GL2Fp}.

A {\bf reflection group} is a finite subgroup of $\GL_n(K)$ generated by reflections. We point out that this is not a property of an abstract finite group, but really a property of a group together with a faithful representation by matrices. Over the complex numbers, all finite reflection groups have been classified by Shephard and Todd \cite{shephardtodd}. We follow their notation for the groups, which we now review for the groups of interest in this paper.

The class of reflection groups that we study in this paper are denoted $G(m,r,n)$ where $m,r,n$ are positive integers and $r$ divides $m$. The group $G(m,1,n)$ consists of $n \times n$ matrices such that every row and column contains at most 1 nonzero entry, and such that every nonzero entry is an $m$th root of unity (we implicitly assume that $K$ contains primitive $m$th roots of unity, i.e., the equation $t^m = 1$ has $m$ distinct solutions in $K$; note that this implies that the characteristic of $K$ does not divide $m$). The group $G(m,r,n)$ is the subgroup of $G(m,1,n)$ defined by the property that the product of all nonzero entries is an $(m/r)$th root of unity.

For example, the group $G(1,1,n)$ is isomorphic to the symmetric group on $n$ letters (which we denote by $\Sigma_n$), and the group $G(m,m,2)$ is isomorphic to the dihedral group of a regular $m$-gon.

\subsection{Specht modules and Garnir polynomials}

Representations of $G(m,r,n)$ are constructed from the representations of the symmetric groups of smaller size known as Specht modules. This construction is described in \cite[\S 5]{kerber}. When the characteristic of $K$ is $0$, Specht modules give all irreducible representations, but in positive characteristic Specht modules are generally reducible. Our reference for this background information is \cite{peel}. We omit the abstract construction of Specht modules since only their realization using Garnir polynomials is relevant for this paper.

Specht modules are indexed by partitions $\lambda$ of $n$. For a given partition $\lambda$ of $n$, a {\bf Young tableau} is a filling of the partition with the numbers from $1$ through $n$. An example of a Young tableau for the partition $(4,2,1)$ follows:
\[
\small \tableau[scY]{1&2&5&7 \cr 3&4 \cr 6}
\]
A {\bf standard Young tableau} is a Young tableau in which the entries in the rows and columns are increasing top to bottom and left to right. For example, the tableau above is standard. The {\bf Garnir polynomial} for a Young tableau $T$ of shape $\lambda$ is defined as follows. Let $a_{i,j}$ be the entry in the $i$th row and $j$th column of $T$. Then
\begin{equation*}
f_T(x) = \prod_{1 \leq d \leq \lambda_1} \prod_{r < s} (x_{a_{r,d}} - x_{a_{s,d}})
\end{equation*} 
is the Garnir polynomial for the Young tableau $T$. For example, the Garnir polynomial for the above tableau is $(x_1-x_3)(x_1-x_6)(x_3-x_6)(x_2-x_4)$. 

The linear span of all Garnir polynomials for tableaux associated to a fixed partition $\lambda$ is the Specht module indexed by $\lambda$, and it is denoted by $S_{\lambda}$. The Garnir polynomials indexed by the standard Young tableaux of shape $\lambda$ form a basis for $S_\lambda$. The degree of these polynomials is $n(\lambda) = \sum_i (i-1)\lambda_i$ and in characteristic $0$, it is known that this is the minimal degree occurrence of this representation in the symmetric algebra \cite[\S 3]{stembridge}.

\subsection{Rational Cherednik algebras}

Let $G \subset \GL(\fh)$ be a reflection group where $\fh$ is a vector space over a field $K$. Let $\cS$ be the set of reflections in $G$. For each $s \in \cS$, we pick a vector $\alpha_s \in \fh^*$ that spans the image of $1-s$, and let $\alpha^\vee_s \in \fh$ be defined by the property
\[
(1-s)x = (\alpha^\vee_s, x)\alpha_s.
\]
Pick $\hbar \in K$ and $c_s \in K$ for each $s \in \cS$, where we require that $c_s = c_{s'}$ if $s$ and $s'$ are conjugate. Let $T(\fh \oplus \fh^*)$ be the tensor algebra on $\fh \oplus \fh^*$. The {\bf (rational) Cherednik algebra} $\bH_{\hbar, c}(G,\fh)$ is the quotient of $K[G] \ltimes T(\fh \oplus \fh^*)$ by the relations
\begin{align} \label{eqn:rca-rel}
[x,x'] = 0, \quad [y,y'] = 0, \quad [y,x] = \hbar(y,x) - \sum_{s \in \cS} c_s (y, \alpha_s) (x, \alpha_s^\vee)s
\end{align}
where $x,x' \in \fh^*$ and $y,y' \in \fh$. The algebra $\bH_{\hbar, c}(G, \fh)$ possesses a $\bZ$-grading: we set $\deg(x) = 1$ for $x \in \fh^*$, $\deg(y) = -1$ for $y \in \fh$, and $\deg(g) = 0$ for $g \in K[G]$.

We have the following PBW-type decomposition of the rational Cherednik algebra:
\begin{align} \label{eqn:pbw}
\bH_{\hbar, c}(G, \fh) = \Sym(\fh) \otimes_K K[G] \otimes_K \Sym(\fh^*)
\end{align}
(see \cite[\S 3.2]{etingofma} in the case when the characteristic of $K$ is 0; the proof holds in general though).

We have isomorphisms $\bH_{\hbar, c}(G, \fh) \cong \bH_{\alpha \hbar, \alpha c}(G, \fh)$ when $\alpha \ne 0$, so we may assume that either $\hbar = 0$ or $\hbar = 1$ without loss of generality. When the parameters $c_s$ are algebraically independent over the prime subfield of $K$ (which is either $\bZ/p$ or $\bQ$), we will say that they are {\bf generic}. In this paper, we will be mostly concerned with the case when the parameters are generic and $\hbar = 0$. In this case, we simply write $\bH(G)$ for the Cherednik algebra. We define $\overline{c}_s = c_{s^{-1}}$. 

\subsection{Category $\cO$}

Following \cite[\S 2.7]{balagovic}, we define $\cO$ to be the category of $\bZ$-graded $\bH_{\hbar, c}(G, \fh)$-modules which are finite dimensional as $K$-vector spaces.

One can construct lowest-weight representations of $\bH_{\hbar,c}(G,\fh)$ in the following way. Let $\tau$ be a representation of $G$. We let $\Sym(\fh)$ act as 0 on $\tau$ and construct the \textbf{Verma module}
\[
M_{\hbar, c}(G,\fh,\tau) = \bH_{\hbar,c}(G,\fh) \otimes_{K[G]\ltimes\Sym(\fh)}\tau.
\]
By the PBW decomposition \eqref{eqn:pbw}, we have
\begin{align} \label{eqn:vermapbw}
M_{\hbar, c}(G,\fh,\tau) = \Sym(\fh^*) \otimes_K \tau
\end{align}
as a $K$-vector space. This is $\bZ$-graded, but does not belong to category $\cO$ since it is infinite-dimensional as a $K$-vector space. There is an intermediate construction, the baby Verma module, which is a quotient of $M_{\hbar, c}(G,\fh,\tau)$, and belongs to $\cO$. We refer to \cite[\S 2.6]{balagovic} for a discussion.

\begin{proposition} \label{prop:betaform}
$M_{\hbar, c}(G,\fh,\tau)$ has a unique maximal graded proper submodule $J_{\hbar, c}(G,\fh, \tau)$, which may be realized as the kernel of the contravariant form 
\[
\beta_{c} \colon M_{\hbar, c}(G,\fh,\tau) \otimes_K M_{\hbar, \overline{c}}(G,\fh^{*},\tau^{*})\to K.
\]
The form $\beta_c$ can be characterized by the property that for all $x \in \fh^*$, $y \in \fh$, $f_1 \in M_{\hbar, c}(G,\fh,\tau)$, $f_2 \in M_{\hbar, \overline{c}}(G,\fh^*,\tau^*)$, $v \in \tau$, and $w \in \tau^*$, we have:
\begin{compactenum}[\rm (1)]
\item $\beta_{c}(xf_1,f_2)=\beta_{c}(f_1,xf_2)$,
\item $\beta_{c}(f_1,yf_2)=\beta_{c}(yf_1,f_2)$,
\item $\beta_{c}(v,w)=w(v)$.
\end{compactenum}
\end{proposition}

See \cite[\S 2.5]{balagovic}. In particular, the quotient 
\[
L_{\hbar, c}(G,\fh,\tau) = M_{\hbar, c}(G,\fh,\tau) / J_{\hbar, c}(G,\fh,\tau)
\]
is an irreducible $\bZ$-graded representation of $\bH_{\hbar, c}(G, \fh)$. In fact, it is finite-dimensional, since it is a quotient of the baby Verma module, so it belongs to category $\cO$.

\begin{proposition}
Every irreducible object of $\cO$ is isomorphic to $L_{\hbar, c}(G,\fh,\tau)$ for some irreducible representation $\tau$ of $G$.
\end{proposition}

When it is clear from context, we will often omit $G$ and $\fh$ from the notation.

Given a submodule $J' \subset M_{\hbar, c}(G, \fh, \tau)$, the following lemma is useful for determining when $J' = J_{\hbar, c}(G, \fh, \tau)$, i.e., when $M_{\hbar, c}(G, \fh, \tau) / J'$ is irreducible. We will use it throughout the paper.

\begin{lemma} \label{lem:socle}
With the notation above, let $J' \subset M_{\hbar, c}(G, \fh, \tau)$ be a graded  $\bH_{\hbar, c}(G, \fh)$-submodule. Consider $N = M_{\hbar, c}(G, \fh, \tau) / J'$ as a $\Sym(\fh^*)$-module and it assume that it is a finite-dimensional representation. Assume that 
\begin{compactenum}[\rm (1)]
\item the socle of $N$ is concentrated in top degree,
\item the socle is irreducible as a representation of $G$,
\item there exists $v$ in the socle of $N$ such that the linear form $\beta_c(v,-)$ is nonzero.
\end{compactenum}
Then $J'=J_{\hbar, c}(G, \fh, \tau)$ and $N$ is an irreducible $\bH_{\hbar, c}(G, \fh)$-module.
\end{lemma}

\begin{proof}
Since $N$ is finite-dimensional, any nonzero $\bH_{\hbar, c}(G)$-submodule of $N$ must nontrivially intersect its socle when considered as a $\Sym(\fh^*)$-module. By Proposition~\ref{prop:betaform}, the maximal proper graded submodule of $N$ is the kernel of $\beta_c \colon N \to \hom_K(M_{\hbar, \ol{c}}(G, \fh^*, \tau^*), K)$. Since $\beta_c$ is $G$-equivariant, conditions (2) and (3) imply that the socle does not intersect $\ker \beta_c$ nontrivially. Hence $\ker \beta_c = 0$ and $N$ is irreducible.
\end{proof}

\subsection{Dunkl operators}

The action of $\bH_{\hbar,c}(G,\fh)$ on the Verma module $\bH_{\hbar,c}(G,\fh) \otimes_{K[G] \ltimes \Sym(\fh)}\tau$ is by left multiplication. However, by \eqref{eqn:vermapbw}, we can write the Verma module as $\Sym(\fh^*) \otimes_K \tau$, and we can explicitly describe the action of the Cherednik algebra in this form. The action of $\Sym(\fh^*)$ is by left multiplication and the action of $G$ is via the diagonal action. For $y \in \fh$, define the {\bf Dunkl operator} $D_y$ on the Verma module $M_c(\tau) = \Sym(\fh^*) \otimes \tau$ by
\[
D_y(f \otimes v) = \hbar \partial_y f \otimes v - \sum_{s \in \cS} c_s \frac{ ( y, \alpha_s )}{\alpha_s} (1-s). f \otimes s.v.
\]
Then the map $y \mapsto D_y$ is the desired action of $\Sym(\fh)$ on $\Sym(\fh^*) \otimes_K \tau$. If we have chosen dual bases $x_1, \dots, x_n \in \fh^*$ and $y_1, \dots, y_n \in \fh$ (which we will below), then we write $D_i$ instead of $D_{y_i}$.

\subsection{Free resolutions}

We will make some use of the theory of free resolutions over polynomial rings, so we briefly review this now. For a more thorough treatment, we refer the reader to \cite[Chapters 17--21]{eisenbud}.

Let $A = K[x_1, \dots, x_n]$ be a polynomial ring in $n$ variables, which we treat as a graded ring with $\deg(x_i) = 1$ (we have in mind $A = \Sym(\fh^*)$ for our applications). Given a graded $A$-module $M$, a {\bf free resolution} of $M$ is a complex of free $A$-modules
\[
\bF_\bullet \colon \cdots \to \bF_i \to \bF_{i-1} \to \cdots \to \bF_1 \to \bF_0 \to 0
\]
such that $\rH_0(\bF_\bullet) = M$ and $\rH_i(\bF_\bullet) = 0$ for $i>0$. We say that $\bF_\bullet$ is {\bf minimal} if all of its differentials become identically 0 after doing the substitution $x_1 \mapsto 0, \dots, x_n \mapsto 0$. Minimal free resolutions exist and are unique up to isomorphism of complexes. We will only consider minimal free resolutions in this paper. It will be convenient to assume that the differentials are degree-preserving, and for that, we introduce the notation $A(-d)$ to denote a free module of rank 1 generated in degree $d$, i.e., $A(-d)_e = A_{e-d}$. Hence each $\bF_i$ is a direct sum of various $A(-j)$. The multiplicity of $A(-j)$ in $\bF_i$ is denoted $\beta_{ij} = \beta_{ij}(M)$ and these are the {\bf graded Betti numbers} of $M$. When displaying them, we will follow {\tt Macaulay 2} notation:
\[
\begin{matrix}
\vdots \\
\beta_{00} & \beta_{11} & \beta_{22} & \cdots \\
\beta_{01} & \beta_{12} & \beta_{23} & \cdots \\
\vdots
\end{matrix}
\]

Let $\pdim M$ be the length of the  minimal free resolution of $M$. Then $\pdim M \le n$ for all $M$, and in fact, we have the Auslander--Buchsbaum formula
\begin{align} \label{eqn:AB}
n = \pdim M + \depth M,
\end{align}
where $\depth M$ is the maximum length of a regular sequence on $M$. In particular, if $I \subset A$ is an ideal whose solution set (over an algebraic closure of $K$) has codimension $c$, then $A/I$ is Cohen--Macaulay if and only if $\pdim A/I = c$. In this case, $A/I$ is Gorenstein if and only if $\rank \bF_c = 1$. If this happens, we immediately get that $\rank \bF_i = \rank \bF_{c-i}$ for all $i$ (this equality is also compatible with the grading, but we won't make much use of it). As a weakening of the Gorenstein property, we say that a Cohen--Macaulay algebra $A/I$ is {\bf level} if $\beta_{c,d} \ne 0$ for exactly one value of $d$. This is equivalent to the property that the socle of any Artinian reduction of $A/I$ (i.e., finite-dimensional quotient of $A/I$ by a maximal length regular sequence) is concentrated in top degree.

\section{Subspace arrangements} \label{sec:subspace}

Fix nonnegative integers $i,n$ and pick a subset $S \subseteq \{1, \dots, n\}$ of size $n-i$ and write $S = \{j_1, \dots, j_{n-i}\}$. For $m \ge 1$, we define
\begin{align*}
X_S^{(m)} &= \{(x_1, \dots, x_n) \in K^n \mid x_{j_1}^m = x_{j_2}^m = \cdots = x_{j_{n-i}}^m\},\\
X^{(m)}_i &= \bigcup_{\substack{S \subseteq \{1, \dots, n\}\\ \#S=n-i}} X_S^{(m)}.
\end{align*}
Let $I_S^{(m)}$ be the ideal generated by the equations $x_{j_1}^m - x_{j_2}^m, x_{j_2}^m - x_{j_3}^m, \dots, x_{j_{n-i-1}}^m - x_{j_{n-i}}^m$ and let $I^{(m)}_i = \bigcap_S I_S^{(m)}$. When $m$ is not divisible by the characteristic of $K$, $I_S^{(m)}$ is the radical ideal of polynomials vanishing on $X_S^{(m)}$. We will only be interested in such cases in this paper, but for this section, we ignore this restriction on the characteristic since one can make uniform statements.

When $m=1$, the ideal $I^{(1)}_i$ is generated by a space of Garnir polynomials of shape $\lambda$, where $\lambda$ can be described as follows. First, write $n = q(n-i-1) + r$ where $0 \le r < n-i-1$. Then $\lambda = (n-i-1, \dots, n-i-1, r)$ where $n-i-1$ appears $q$ times \cite[Corollary 2.3]{lili}. For general $m$, we apply the substitution $x_j \mapsto x_j^m$ to transform $I^{(1)}_i$ into $I^{(m)}_i$.

\begin{proposition}[Etingof--Gorsky--Losev] \label{prop:CMkequals}
The affine variety $X_i^{(m)}$ is Cohen--Macaulay in characteristic $0$ (and hence for sufficiently large characteristic) when $2i<n$.
\end{proposition}

See \cite[Proposition 3.9]{EGL}. This result is only stated for $m=1$, but we can use Remark~\ref{rmk:reducem=1} to reduce to this case. In \cite{EGL}, the variety $X_i^{(1)}$ is denoted $X_{[n/(n-i)]}$. We remark that in characteristic $0$, the minimal free resolution (which is in fact a complex of modules over the Cherednik algebra for suitable parameters) of the coordinate ring of $X_i^{(1)}$ is constructed in \cite[\S 5]{BGS}.

Now we calculate the Hilbert series of $K[x_1, \dots, x_n] / I_i^{(m)}$. 

Call a subgroup $M'$ of a finitely generated free Abelian group $M$ {\bf saturated} if $M/M'$ is free. This implies that we can give $M'$ the structure of a closed subscheme of $M$ by defining its coordinate ring to be $\Sym(M^*) / I$ where $I$ is the linear ideal generated by $(M/M')^*$ (here $(-)^* = \hom_\bZ(-,\bZ)$). Given a collection $Y$ of saturated subgroups of a finitely generated free Abelian group, we can view $Y$ as a scheme over $\bZ$. For a field $K$, let $Y(K)$ be the base change of $Y$ to $K$.

\begin{lemma} \label{lem:flatsubspaces}
Let $Y$ be a collection of saturated subgroups of $\bZ^n$, viewed as a scheme. Then the Hilbert series of $Y(K)$ is independent of the field $K$, i.e., $Y$ is flat over $\bZ$.
\end{lemma}

\begin{proof}
Let $Y = Y_1 \cup \cdots \cup Y_N$ be the subgroups in $Y$. Let $A = \bZ[x_1, \dots, x_n]$ and let $I_j \subset A$ be the ideal of $Y_j$ (which is generated by linear forms). Then $A/I_j$ is a free $\bZ$-module for each $j$. The ideal of $Y$ is defined by $I_1 \cap \cdots \cap I_N$, and we have an injection
\[
0 \to A / (I_1 \cap \cdots \cap I_N) \to A/I_1 \oplus \cdots \oplus A/I_N,
\]
which implies that $A/(I_1 \cap \cdots \cap I_N)$ is a free $\bZ$-module. In particular, the dimensions of the graded pieces of $A/(I_1 \cap \cdots \cap I_N) \otimes K$ are independent of $K$, which finishes the proof.
\end{proof}

For the purposes of the paper, we will be most interested in the variety $X^{(m)}_i$ when $2i < n$. In particular, the ideal $I_i^{(m)}$ is generated by Garnir polynomials associated to the partition of shape $(n-i-1,i+1)$ if $n \ge 2i+2$, or of shape $(i,i,1)$ if $n=2i+1$.

\begin{proposition}
If $n \ge 2i+2$, then the Hilbert series of $X_i^{(1)}$ is
\[
\frac{\sum_{j=0}^i \binom{n-i+j-2}{j} t^j + \binom{n-1}{i-1} t^{i+1}}{(1-t)^{i+1}}.
\]
If $n=2i+1$, then the Hilbert series of $X_i^{(1)}$ is 
\[
\frac{\sum_{j=0}^{i+1} \binom{n-i+j-2}{j} t^j}{(1-t)^{i+1}}.
\]
\end{proposition}

\begin{proof}
Set $X = X_i^{(1)}$. It is clear that each subspace in $X$ comes from a saturated subgroup: we are just setting certain coordinates equal to one another. So by Lemma~\ref{lem:flatsubspaces}, the Hilbert series of $X$ is independent of the field that we work over. So we can calculate it by working over a field of characteristic $0$, so for the remainder of the proof we set $K = \bQ$. In this case, we know by Proposition~\ref{prop:CMkequals} that the coordinate ring of $X$ is Cohen--Macaulay. This implies that the coefficients of the numerator polynomial are nonnegative.

First suppose that $n \ge 2i+2$. Since the ideal of $X$ is generated by polynomials of degree $i+1$, the Hilbert function of $X$ agrees with that of $\bQ[x_1, \dots, x_n]$ in degrees $\le i$, so its Hilbert series is of the form
\[
\frac{(\sum_{j=0}^{i} \binom{n-i + j-2}{j} t^j) + t^{i+1}Q(t)}{(1-t)^{i+1}}.
\]
The sum counts the number of monomials of degree at most $i$ in $n-i-1$ variables, which (by homogenization) is the same as the number of monomials of degree exactly $i$ in $n-i$ variables, and this number is $\binom{n-1}{i}$. Plugging in $t=1$ into the numerator gives $\deg X$, which we know is $\binom{n}{i}$ since $X$ is a union of $\binom{n}{i}$ linear subspaces of the same dimension. So $Q(1) = \binom{n}{i} - \binom{n-1}{i} = \binom{n-1}{i-1}$. If we had no equations of degree $i+1$, then the coefficient of $t^{i+1}$ would be $\binom{n-1}{i+1}$, but we have added 
\[
\dim S_{(n-i-1,i-1)} = \frac{n!(n-2i-1)}{(i+1)!(n-i)!} = \binom{n-1}{i+1} - \binom{n-1}{i-1}
\]
many equations, so we conclude that the coefficient of $t^{i+1}$ is $\binom{n-1}{i-1}$. Since the coefficients of $Q(t)$ are nonnegative, we conclude that $Q(t)$ must be a constant equal to $\binom{n-1}{i-1}$.

Now we consider the case $n=2i+1$. In this case, the Garnir polynomials live in degree $i+2$. The sum $\sum_{j=0}^{i+1} \binom{n-i+j-2}{j}$, by the previous reasoning, is $\binom{n}{i-1}$, which is the same as $\deg X$. Again by the Cohen--Macaulay property, the coefficients of the numerator of the Hilbert series must be nonnegative, so this implies that the claimed Hilbert series is correct.
\end{proof}

\begin{lemma} \label{lem:powerhilb}
Let $I \subset A = K[x_1, \dots, x_n]$ be an ideal generated by $f_1, \dots, f_r$ and let $I^{(m)}$ be the ideal generated by $f_1(x_1^m, \dots, x_n^m), \dots, f_r(x_1^m, \dots, x_n^m)$. If the Hilbert series of $I$ is $P(t)/(1-t)^d$, then the Hilbert series of $I^{(m)}$ is 
\[
\frac{P(t^m)(1+t+\cdots+t^{m-1})^{n-d}}{(1-t)^d}.
\]
\end{lemma}

\begin{proof}
Let $A^{(m)} \subset A$ be the subring generated by $x_1^m, \dots, x_n^m$. Let $J \subset A^{(m)}$ be the ideal in $A^{(m)}$ generated by $f_1(x_1^m, \dots, x_n^m), \dots, f_r(x_1^m, \dots, x_n^m)$. Then the substitution $g(x_1, \dots, x_n) \mapsto g(x_1^m, \dots, x_n^m)$ gives an isomorphism from $I$ to $J$ which multiplies degrees by $m$. Hence the Hilbert series of $A^{(m)}/J$ is $P(t^m) / (1-t^m)^d$. 

We have $A/I^{(m)} = (A^{(m)}/J) \otimes_{A^{(m)}} A$. Since $A$ is a free $A^{(m)}$-module, and the degrees of the basis elements are given by the generating function $(1+t+ \cdots + t^{m-1})^n$, we conclude that the Hilbert series of $A/I^{(m)}$ is
\[
\frac{P(t^m)}{(1-t^m)^d} \frac{(1-t^m)^n}{(1-t)^n} = \frac{P(t^m)(1+t+\cdots+t^{m-1})^{n-d}}{(1-t)^d}. \qedhere
\]
\end{proof}

\begin{remark}
In the above lemma, we could replace $x_i^m$ by any degree $m$ homogeneous polynomials $g_1, \dots, g_n$ which form a regular sequence.
\end{remark}

\begin{corollary}
If $n \ge 2i+2$, then the Hilbert series of $A/I_i^{(m)}$ is
\[
\frac{(\sum_{j=0}^i \binom{n-i+j-2}{j} t^{mj} + \binom{n-1}{i-1} t^{m(i+1)})(1+t +\cdots + t^{m-1})^{n-i-1}}{(1-t)^{i+1}}.
\]
If $n =2i+1$, then the Hilbert series of $A/I_i^{(m)}$ is 
\[
\frac{(\sum_{j=0}^{i+1} \binom{n-i+j-2}{j} t^{mj})(1+t+\cdots+t^{m-1})^i}{(1-t)^{i+1}}.
\]
\end{corollary}

\begin{proposition} 
The affine variety $X_1^{(m)}$ is Gorenstein.
\end{proposition}

\begin{proof} 
First assume that $m=1$. Note that $X_1^{(1)}$ is the intersection of $n$ planes $Y_i$ where $Y_i$ is the plane given by the conditions $x_j = x_k$ when $i \ne j$ and $i \ne k$. Let $J_i$ be the ideal defining $Y_i$. So we can write the coordinate ring of $X$ as
\[
K[X_1^{(1)}] = \left(\frac{K[x_1 - x_2, x_1 - x_3, \dots, x_1 - x_n]}{J_1 \cap \cdots \cap J_n}\right)[x_1 + \cdots + x_n].
\]
Write $z_i = x_1 - x_{i+1}$. We will show that $K[z_1, \dots, z_{n-1}] / (J_1 \cap \cdots \cap J_n)$ is Gorenstein. If we projectivize, we get $n$ points in $\bP^{n-2}$, where the projectivization of $Y_1$ is $[1 : 1 : \cdots : 1]$, and the projectivization of $Y_i$ for $i>1$ is $[0 : \cdots : 1 : \cdots : 0]$ which has a 1 in the $i$th spot. These are in {\it linearly general position}, (any $d$-dimensional linear subspace of $\bP^{n-2}$ with $d < n-2$ contains at most $d+1$ of the points). Hence the ring $K[z_1, \dots, z_{n-1}] / (J_1 \cap \cdots \cap J_n)$ is Gorenstein by \cite[Theorem 5]{cayleybacharach} or \cite[Theorem 7.2]{popescu}.

In particular, the minimal free resolution of $K[X^{(1)}_1]$ is self-dual. The case of general $m$ follows from $m=1$ from the substitution $x_i \mapsto x_i^m$: $x_1^m, \dots, x_n^m$ forms a regular sequence, so the minimal free resolution of $K[X^{(m)}_1]$ is self-dual, which implies that it is Gorenstein. 
\end{proof}

\begin{conjecture} When ${\rm char}(K)>i$ and $n>2i$, $X_i^{(m)}$ is Cohen--Macaulay and its coordinate ring is a level algebra, i.e., the last term in its minimal free resolution is generated in a single degree.
\end{conjecture}

The conjecture is false without the assumption on ${\rm char}(K)$, see \cite[Example 5.2]{BGS}.

\begin{remark} \label{rmk:reducem=1}
We can reduce to the case $m=1$ by a formal argument: namely, given a homogeneous ideal $I \subset A = K[x_1, \dots, x_n]$ such that $A/I$ is Cohen--Macaulay, and a regular sequence of homogeneous polynomials of positive degree $f_1, \dots, f_n \in B = K[y_1, \dots, y_N]$, let $\phi(x_i) = f_i$. Then $B/\phi(I)$ is also Cohen--Macaulay. For our purposes, we take $A = B$ and $f_i = x_i^m$.
\end{remark}

\section{Characters for $G(m,1,n)$ in the non-modular case} \label{sec:Gm1n}

Let $G = G(m,1,n) = \Sigma_n \ltimes (\bZ/m)^n$. In this section, we calculate the characters of the Cherednik algebra for $G$ with generic parameters (we will handle $\hbar = 0$ and $\hbar \ne 0$) in the case when the characteristic of the ground field does not divide the size of $G$. The techniques are the same as the techniques in \cite[\S 6]{gordon}, but we reproduce the arguments for completeness.

Let $\ul{\lambda} = (\lambda^0, \dots, \lambda^{m-1})$ be an $m$-tuple of partitions such that $|\ul{\lambda}| := \sum_i |\lambda^i| = n$ (the indices are elements of $\bZ/m$). These $\ul{\lambda}$ naturally index the conjugacy classes of $G$, and also the complex irreducible representations $S_{\ul{\lambda}}$ with character $\chi^{\ul{\lambda}}$ \cite[\S 5]{kerber}. By general principles, these also index the irreducible representations over an algebraically closed field of characteristic $p$ whenever $p$ does not divide the size of $G$ (the non-modular case).

Let ${\ul{\lambda}}^* = (\lambda^0, \lambda^{-1}, \dots, \lambda^{-m+1})$ denote the partition indexing the dual representation of $S_{\ul{\lambda}}$. Set
\begin{align*}
n({\ul{\lambda}}) &= \sum_{i=0}^{m-1} \bigg( i|\lambda^i| + \sum_j (j-1)\lambda^i_j \bigg),\\
(t)_n &= (1-t)(1-t^2) \cdots (1-t^n).
\end{align*}
Given a partition $\lambda$ and a box $s \in \lambda$ in its Young diagram, let ${\rm hook}(s)$ denote its hook length. Set
\begin{align*}
H_\lambda(t) &= \prod_{s \in \lambda} (1-t^{{\rm hook}(s)}), \qquad
H_{\ul{\lambda}}(t) = \prod_{i=0}^{m-1} H_{\lambda^i}(t).
\end{align*}
Let $\xi \in K$ be a primitive $m$th root of unity. Let $z_{\ul{\lambda}}$ be the size of the stabilizer subgroup of any element in the conjugacy class of ${\ul{\lambda}}$. Define
\begin{align} \label{eqn:wreathkostka}
K'_{{\ul{\mu}},{\ul{\lambda}}}(t) = 
H_{\ul{\lambda}}(t^m) \sum_{\ul{\rho}} \frac{\chi^{\ul{\lambda}}({\ul{\rho}}) \chi^{{\ul{\mu}}^*}({\ul{\rho}})}{z_{\ul{\rho}} \prod_{i,j} (1-\xi^i t^{\rho^i_j})}.
\end{align}
Finally, by \cite[(5.5)]{stembridge}, the generating function for the occurrences of $\chi^{\ul{\lambda}}$ in the coinvariants algebra for $G$ acting on its reflection representation is
\begin{align} \label{eqn:coinvariants}
f_{\ul{\lambda}}(t) = \frac{t^{n({\ul{\lambda}})} (t^m)_n}{H_{\ul{\lambda}}(t^m)}.
\end{align}

\begin{proposition} 
Consider $p$ not dividing $m^n n!$, and $\tau = S_{\ul{\lambda}}$. When $\hbar=0$, the $G$-equivariant Hilbert series of $L_c(\tau)$ is
\[
\sum_{\ul{\mu}} K'_{{\ul{\mu}}, {\ul{\lambda}}}(t) [S_{\ul{\mu}}].
\]
In particular, the usual Hilbert series is
\[
\dim(\tau) \frac{H_{\ul{\lambda}}(t^m)}{(1-t)^n}.
\]
\end{proposition}

\begin{proof} Let $p_{{\ul{\lambda}}, {\ul{\mu}}}(t) = \sum_i [L(S_{\ul{\lambda}})_i : S_{\ul{\mu}}] t^i$. Following \cite[\S 6.4]{gordon} (here we need to know that the Calogero--Moser space for $G$ is nonsingular, which follows from \cite[Corollary 1.14]{calogeromoser}), we can write
\[
[M(S_{\ul{\lambda}})] = \sum_{\ul{\mu}} t^{-n({\ul{\lambda}}^*)} f_{{\ul{\lambda}}^*}(t) p_{{\ul{\lambda}}, {\ul{\mu}}}(t) [S_{\ul{\mu}}].
\]
By \cite[(9)]{gordon}, we know that 
\[
[M(S_{\ul{\lambda}})] = \sum_{\ul{\mu}} f_{{\ul{\mu}}}(t) [S_{\ul{\mu}}] [S_{\ul{\lambda}}].
\]
Hence by \cite[4.4]{gordon}, it is enough to show that
\begin{align} \label{eqn:wreathtoshow} 
\sum_{\ul{\mu}} f_{{\ul{\mu}}}(t) \chi^{\ul{\mu}}({\ul{\rho}}) \chi^{\ul{\lambda}}({\ul{\rho}}) 
= \sum_{\ul{\mu}} t^{-n({\ul{\lambda}}^*)} f_{{\ul{\lambda}}^*}(t) K'_{{\ul{\mu}}, {\ul{\lambda}}}(t) \chi^{\ul{\mu}}({\ul{\rho}})
\end{align}
for all ${\ul{\rho}}$. Inverting \eqref{eqn:wreathkostka}, we see that
\begin{align} \label{eqn:wreathkostkainverse} 
\sum_{\ul{\mu}} K'_{{\ul{\mu}}, {\ul{\lambda}}}(t) \chi^{\ul{\mu}}({\ul{\rho}}) 
= \frac{H_{\ul{\lambda}}(t^m)}{\prod_{i,j} (1-\xi^it^{\rho^i_j})} \chi^{\ul{\lambda}}({\ul{\rho}}).
\end{align}
Also, by \cite[(2.5)]{stembridge}, we have
\begin{align} \label{eqn:regularrep}
\sum_{\ul{\mu}} f_{\ul{\mu}}(t) \chi^{\ul{\mu}}({\ul{\rho}}) 
= \frac{(t^m)_n}{\det(1-t{\ul{\rho}})} = \frac{(t^m)_n}{\prod_{i,j}(1- \xi^it^{\rho^i_j})}
\end{align}
where in $\det(1-t{\ul{\rho}})$, we use $\ul{\rho}$ to mean any element in the conjugacy class of ${\ul{\rho}}$ acting on the reflection representation. Now \eqref{eqn:wreathtoshow} follows from \eqref{eqn:coinvariants}, \eqref{eqn:wreathkostkainverse}, and \eqref{eqn:regularrep}. The statement about the usual Hilbert series follows by considering ${\ul{\rho}} = (n, \emptyset, \dots, \emptyset)$ in \eqref{eqn:wreathkostkainverse}.
\end{proof}

Let $A = K[x_1, \dots, x_n]$ and $A^{(p)} = K[x_1^p, \dots, x_n^p]$. Then $A$ is a free $A^{(p)}$-module, and $A = A^{(p)} \otimes Q$ for some graded $G$-representation $Q$. Write $[Q] = \sum_i [Q_i] t^i$ for its $G$-equivariant Hilbert series.

\begin{proposition} 
Consider $p$ not dividing $m^n n!$ and $\tau = S_{\ul{\lambda}}$. When $\hbar=1$, the $G$-equivariant Hilbert series of $L_c(\tau)$ is
\[ 
[Q] \sum_{\ul{\mu}} K'_{{\ul{\mu}}, {\ul{\lambda}}}(t) [S_{\ul{\mu}}].
\]
In particular, the usual Hilbert series is
\[
\dim(\tau) \frac{H_{\ul{\lambda}}(t^{mp})}{(1-t)^{n}}.
\]
\end{proposition}

\begin{proof}
We define baby Verma modules $M^\circ(S_{\ul{\lambda}})$ as in \cite[\S 4]{gordon}, but instead of using the ideal generated by the positive degree invariants of $G$, we use the ideal generated by the $p$th powers of these invariants. The result follows formally from \cite[\S 6.4]{gordon} once we show that 
\begin{align} \label{eqn:vermamult} 
\sum_j [M^\circ(S_{\ul{\lambda}})[j] : L(S_{\ul{\lambda}})] t^{-j} = t^{-n({\ul{\lambda}})p} f_{{\ul{\lambda}}}(t^p).
\end{align}
Ignoring the grading, one has a $G$-equivariant isomorphism
\begin{align} \label{eqn:eqvtstructure} 
Q \otimes K[G] \cong L(S_{\ul{\lambda}})
\end{align}
for all ${\ul{\lambda}}$ because of the existence of an Azumaya algebra on the corresponding Calogero--Moser space (see \cite[Remark 1.2.3]{bfg} for the case $m=1$) and the rigidity of $G$-modules (see \cite[Proof of Theorem 1.7]{calogeromoser} for details).

Let $L(S_{\ul{\lambda}})'$ be the limit of $L(S_{\ul{\lambda}})$ as $c \to 0$. Then $L(S_{\ul{\lambda}})'$ is a $G$-equivariant module over the Weyl algebra, hence is of the form $Q \otimes V$ for some graded $G$-representation $V$ (forgetting the grading, $V$ is the regular representation), and this identification respects the grading and $G$-structure. Hence the $G$-equivariant Hilbert series of $L(S_{\ul{\lambda}})$ is divisible by that of $Q$, which means that \eqref{eqn:vermamult} holds up to a power of $t$ using \eqref{eqn:eqvtstructure} and \cite[\S 5.6]{gordon}. So the desired formula for the usual Hilbert series holds up to a power of $t$. It is correct as stated because the coefficients of $t^i$ agree for $i \le 0$.
\end{proof}

\begin{remark}
For the general class of complex reflection groups $G(m,r,n)$, the Calogero--Moser space is singular in all cases not previously considered, so the techniques used do not apply. However, the Hilbert series of $L_c({\rm triv})$ is given by \cite[Proposition 3.1]{griffeth} in characteristic 0 and $\hbar=0$:
\[
\frac{(1-t^m)(1-t^{2m}) \cdots (1-t^{(n-1)m}) (1-t^{nm/d})}{(1-t)^n}.
\]
The same proof works assuming that $p$ does not divide $m^n n!/d$ (see also Theorem~\ref{thm:paramdegeneration}).
\end{remark}

\section{Degenerating $G(m,r,n)$ to $G(r,r,n)$} \label{sec:degeneration}

Let $G = G(m,r,n)$ and $G' = G(r,r,n)$. Write $q = m/r$. We have a surjection $G \to G'$ by raising each entry of the matrix to the $q$th power. In this section, we set $\hbar = 0$ and take the parameters $c_s$ to be generic. Consider the Cherednik algebra $\bH(G)$ in characteristic $p$ where $p$ does not divide $m$. Given an irreducible representation $\tau'$ of $G'$, we can think of it as an irreducible representation $\tau$ of $G$ via the surjection above. In this section, we will reduce the problem of calculating the character of $L_c(\tau)$ (considered as an irreducible representation of $\bH(G)$) to the problem of calculating the character of $L_c(\tau')$ (considered as an irreducible representation of $\bH(G')$).

Fix a primitive $m$th root of unity $\xi$. We have two types of reflections: 
\begin{compactitem}[$\bullet$]
\item $s_{ij}^k$ ($0 \le k < m$ and $1 \le i < j \le n$): $x_i \mapsto \xi^k x_j$, $x_j \mapsto \xi^{-k} x_i$, and $x_\ell \mapsto x_\ell$ for $\ell \notin \{i,j\}$, 
\item $t^k_i$ ($1 \le k < q$ and $1 \le i \le n$): $x_i \mapsto \xi^{rk} x_i$ and $x_\ell \mapsto x_\ell$ for $\ell \ne i$. 
\end{compactitem}
It is also useful to define $s_{ij}^k$ when $i > j$ in the same way as above, in which case we note that $s_{ij}^k = s_{ji}^{-k}$. When $n>2$ or when $n=2$ and $r$ is odd, the conjugacy classes of $G$ are
\[
\{ s_{ij}^k \mid 0 \le k < m,\ 1 \le i < j \le n \}, \quad \{ t^1_i \mid 1 \le i \le n \}, \quad \dots, \quad \{ t^{q-1}_1 \mid 1 \le i \le n \}.
\]
Let $c_0, \dots, c_{q-1}$ be the parameters of $\bH(G)$ for the above conjugacy classes. In the case $n=2$ and $r$ even, the first set above splits into two conjugacy classes depending on if $k$ is even or odd. In this case, write $c_0^+, c_0^-, c_1, \dots, c_{q-1}$ for the conjugacy classes. To simplify notation, we will write $c_0^{(k)}$ below. This means $c_0$ in the case $n>2$ or $n=2$ and $r$ odd, and when $n=2$ and $r$ is even, it means $c_0^+$ when $k$ is even and $c_0^-$ when $k$ is odd.

Let $\tau'$ be a representation of $G'$, which we extend to a representation $\tau$ of $G$ via the surjection $G \to G'$. Let $M = M(\tau)$ be the Verma module for $\bH(G)$. Let $J^1$ be the kernel of the contravariant form on $M(\tau)$ when viewed as a representation of $\bH(G')$, and let $h^1(t)$ be the Hilbert series of $M/J^1$. Also, let $A = \mathrm{Sym}(\fh^*)$ and define $J^q$ to be the $A$-submodule of $M$ generated by $J^1$ after substituting $x_i^q$ for $x_i$.

Let $D_1, \dots, D_n$ be the Dunkl operators for $\bH(G)$ and $D'_1, \dots, D'_n$ be the Dunkl operators for $\bH(G')$. For the reflection $s = s_{ij}^k$ with $i < j$, we pick $\alpha_s = x_i - \xi^k x_j$ so that $\langle y_i, \alpha_s \rangle = 1$.

\begin{lemma} \label{lem:dunkldeg}
Let $f(x_1, \dots, x_n) = f'(x_1^q, \dots, x_n^q)$ and pick $v \in \tau$. Set $g = D'_i(f' \otimes v)$. Then
\[
D_i(f \otimes v) = qx_i^{q-1} g(x^q).
\]
\end{lemma}

\begin{proof}
By linearity, we can assume that $f$ is a monomial of the form $x_1^{qd_1} \cdots x_n^{qd_n}$ (so $f' = x_1^{d_1} \cdots x_n^{d_n}$). Then $f$ is invariant under $t^k_j$ for $k=1,\dots,q-1$ and $j=1,\dots,n$, and hence are killed by these summands of the $D_i$. Now fix $j$ with $i<j$ and consider the sum 
\[
\sum_{k=0}^{m-1} c_0^{(k)} \frac{(1-s^k_{ij})x_1^{qd_1} \cdots x_n^{qd_n}}{x_i - \xi^{k} x_j} \otimes s_{ij}^k v,
\]
which is a summand of the expression for $D_i(f \otimes v)$. By our assumption that $\tau$ is a representation of $G'$, we have $s^k_{ij} v = s^{qk}_{ij} v$, and we may write instead $s^{qk}_{ij} v$ (where the superscript is understood modulo $m$). If $d_i \ge d_j$, then this sum becomes 
\begin{align} \label{eqn:degn1}
\sum_{k=0}^{m-1} c_0^{(k)} (\prod_{\ell \ne i,j} x_\ell^{qd_\ell} ) (\sum_{\ell=0}^{N} x_i^{qd_i-1-\ell m} x_j^{qd_j + \ell m}) \otimes s_{ij}^{qk} v,
\end{align}
where $N = \lfloor (q(d_i-d_j)-1)/m \rfloor$. On the other hand, consider the sum
\begin{align} \label{eqn:degn2}
g = \sum_{k=0}^{r-1} c_0^{(k)} (\prod_{\ell \ne i,j} x_\ell^{d_\ell}) (\sum_{\ell = 0}^{N} x_i^{d_i-1-\ell r} x_j^{d_j + \ell r}) \otimes s^{qk}_{ij} v,
\end{align}
which is a summand of the expression for $D'_i(f' \otimes v)$. If we apply the substitution $x_\ell \mapsto x_\ell^q$ to \eqref{eqn:degn2} and then multiply by $qx_i^{q-1}$, we get \eqref{eqn:degn1}, which matches up the corresponding summands of $D_i(f \otimes v)$ and $qx_i^{q-1} g(x^q)$. There are three other cases (corresponding to the options $d_i < d_j$ and $i > j$), but they are handled in a similar way, and we omit them.
\end{proof}

\begin{theorem} \label{thm:paramdegeneration} 
$M/J^q$ is an irreducible representation of $\bH(G)$.
\end{theorem}

\begin{proof}
We claim that $J^q$ is an $\bH(G)$-submodule of $M$. Pick $f'(x) \otimes v \in J^1$, and write $f(x) \otimes v = f'(x^q) \otimes v$. The ideal $J^q$ is linearly spanned by elements of the form $x_1^{d_1} \cdots x_n^{d_n}(f \otimes v)$, so to show that $J^q$ is an $\bH(G)$-submodule, it suffices to show that $D_i x_1^{d_1} \cdots x_n^{d_n}(f \otimes v) \in J^q$ for all $i$ and $d_1, \dots, d_n$. We will do this by induction on $d = d_1 + \cdots + d_n$. First suppose $d=0$. Then $D_i(f \otimes v) = qx_i^{q-1} g(x^q)$ where $g = D'_i(f' \otimes v)$ by Lemma~\ref{lem:dunkldeg}. Since $g \in J^1$, this implies that $D_i(f \otimes v) \in J^q$ by the definition of $J^q$. To apply the induction step we use the commutation relation \eqref{eqn:rca-rel}
\[
D_ix_j = x_jD_i - \sum_s c_s (x_j, \alpha_s) (x_i, \alpha^\vee_s) s
\]
and note that $J^q$ is preserved by $G$. This proves the claim, so $J^q$ is an $\bH(G)$-submodule of $M$.

Consider the limit $c_0 \to 0$ (respectively, $c_0^{\pm} \to 0$). Then $\bH(G)$ degenerates to a semidirect product $G' \ltimes \bH((\bZ/q)^n)$ (note that $\bH((\bZ/q)^n)$ has $(q-1)^n$ parameters $c_{i,j}$ for $i=1,\dots,n$ and $1 \le j \le q-1$, and we are considering the case when we have collapsed them to $q-1$ parameters by setting $c_{1,j} = c_{2,j} = \cdots = c_{n,j}$). A direct calculation using Lemma~\ref{lem:socle} shows that $L_c({\rm triv})$ is the quotient of $M_c({\rm triv})$ by the ideal generated by the $q$th powers of the variables. In particular, all irreducible composition factors in $M_c({\rm triv})$ of $\bH((\bZ/q)^n)$ are isomorphic to $L_{c}({\rm triv})$. The Hilbert series of $L_c({\rm triv})$ is $[q]^n$ and it affords the trivial representation for $(\bZ/q)^n$ in lowest degree. So $M/J^q$ is built out of $h^1(1)$ copies of such representations each with lowest degree which is a multiple of $q$. The number of such representations with lowest degree $kq$ is the coefficient of $t^{kq}$ in $h^1(t^q)$. Let the generators be called $f_1, \dots, f_{h^1(1)}$. Then the set
\[
\{x_1^{d_1} \cdots x_n^{d_n} f_i \mid i=1,\dots,h^1(1),\ 0 \le d_j < q \}
\]
forms a basis for $M/J^q$. If there is a proper submodule of $M/J^q$, it is generated by some of the $f_i$.

Write $f_i(x) = f'_i(x^q)$. If $\deg f_i > 0$, we claim that there is some $j$ such that $D_jf_i \ne 0$. We have $D_jf_i(x) = qx_i^{q-1} D^1_j f'_i(x^q)$. But $D^1_j f'_i \ne 0$ for some $j$, and we can write it as a linear combination of the $f_k$. The statement above about the basis for $M/J^q$ implies that $qx_i^{q-1} D^1_j f'_i(x^q) \ne 0$. Thus we see that $M/J^q$ is irreducible as an $\bH(G)$-representation.
\end{proof}

\begin{corollary}
Let $h^1(t)$ be the Hilbert series of $M/J^1$. Then the Hilbert series of $M/J^q$ is $h^1(t^q) \cdot (1+t+\cdots+t^{q-1})^n$.
\end{corollary}

\begin{proof}
This follows by a minor adaptation of Lemma~\ref{lem:powerhilb}.
\end{proof}

\section{The case $\tau = 1$} \label{sec:tau-triv}

For this section, we focus on the case $\tau = 1$. So the Verma module $M(\tau)$ is the polynomial ring $A$.

\subsection{Symmetric groups}

\begin{proposition} \label{prop:dunklspecht} 
If $n \equiv i \pmod p$ where $0 \le i \le p-1$, then the Dunkl operators for $G(m,1,n)$ kill the generators of $I^{(m)}_i$.
\end{proposition}

\begin{proof} 
If $i \ne p-1$ or $i = p-1$ and $n > 2p-1$, then $\lambda = (n-i-1,i+1)$. Otherwise, we have $i=p-1$ and $n = 2p-1$, in which case $\lambda = (p-1,p-1,1)$. 
   
In the first case, we have $n > 2p - 1$ and $\lambda = (n-i-1,i+1)$. For a filling $e$ of the Young diagram for $\lambda$ let $\{\{e_1,e_2\}, \{e_3,e_4\}, \ldots, \{e_{2i+1},e_{2i+2}\}\}$ be the first $i + 1$ columns. Let $f(e)$ be the associated Garnir polynomial after doing the substitution $x_i \mapsto x_i^m$. A generating set of $I_i^{(m)}$ is given by $f(e)$ for all $e$. The reflections $t_i^k$ fix the $f(e)$, so those terms in $D_r f(e)$ are $0$.

If $r \notin \{e_i\}$, then the reflections $s_{r,r'}^k$ with $r' \notin \{e_i\}$ fix $f(e)$, so that term of $D_r$ kills $f(e)$. For all $0 \leq j \leq i$, the terms for the reflections $s_{r,e_{2j+1}}^k$ and $s_{r,e_{2j+2}}^k$ will cancel each other. Otherwise, if $r \in \{e_i\}$, we may as well take $r=e_1$ without loss of generality. We let $g=x_{e_1}^{m-1}(x_{e_3}^m - x_{e_4}^m)\cdots(x_{e_{2i + 1} }^m- x_{e_{2i + 2}}^m)$.  All the reflections generate terms that are multiples of $g$. The total comes to $(n - i)mg$, which is $0$ since $n \equiv i \pmod p$. 

Now consider the case $n = 2p - 1$ and $\lambda = (p-1,p-1,1)$. For a filling $e$ of the Young diagram for $\lambda$ let $\{\{e_1,e_2,e_3\}, \{e_4,e_5\}, \ldots, \{e_{2p-2},e_{2p-1}\}\}$ be the entries in the columns of the diagram and let $f(e)$ be the associated Garnir polynomial after doing the substitution $x_i \mapsto x_i^m$. A generating set for $I_i^{(m)}$ is given by $f(e)$ for all $e$.

To show that $D_i f(e) = 0$, it is enough to consider the case $i=e_1$ and $i=e_4$ by symmetry of the $e_j$. First consider $D_{e_1}$. This ends up being similar to the previous case: we let $g=x_{e_1}^{m-1}(2x_{e_1}^m - x_{e_2}^m - x_{e_3}^m)(x_{e_2}^m - x_{e_3}^m)(x_{e_4}^m - x_{e_5}^m)\cdots(x_{e_{2p-2} }^m- x_{e_{2p-1}}^m)$. All the reflections generate terms that are multiples of $g$ (or add up to a multiple of $g$ when considered in pairs); the terms sum to $mpg \equiv 0 \pmod p$. For the case with $D_{e_4}$, the sum of the terms produced by all the reflections is
\[
pmx_{e_1}^{m-1}(x_{e_1}^m - x_{e_2}^m)(x_{e_1}^m - x_{e_3}^m)(x_{e_2}^m - x_{e_3}^m)(x_{e_6}^m - x_{e_7}^m)\cdots(x_{e_{2p-2} }^m- x_{e_{2p-1}}^m) \equiv 0 \pmod p.
\]
Therefore, all the generators of the ideal are killed by the Dunkl operators.
\end{proof}

We therefore know that $I^{(m)}_i \subset J$. We conjecture that the ideal $J$ is generated by the generators of $I^{(m)}_i$ and a regular sequence on $A/I^{(m)}_i$. In the case where $m=1$, calculations indicate that $x_1+\dots+x_n$ is one of the elements of the regular sequence. When $m=1$ and $n \equiv 1 \pmod p$, computer calculations suggest that the regular sequence is $\{x_1+\dots+x_n, x_{n-1}^p-x_{n-1}x_n^{p-1}+x_n^p\}$.

\subsection{$G(m,m,n)$ for $m>1$, $p$ divides $n$} \label{sec:Gmmn-pdividesn}

Let $G = G(m,m,n)$ in characteristic $p$ where $p$ divides $n$ but not $m$.  

\begin{proposition} \label{prop:Gmmn-pdivides}
The ideal $J$ is generated by the differences of the $m$th powers of the $x_i$ and the squarefree monomials of degree $p$.
\end{proposition}

\begin{proof}
Let $J'$ be the ideal generated by the differences of the $m$th powers of the $x_i$ and the squarefree monomials of degree $p$.   The Dunkl operators in this case can be written as:
\[
D_i f= -c \sum_{\substack{r \neq i, \\0 \leq k \leq m-1}} \frac{(1 - s_{i,r}^k )f}{x_i - \xi^{-k}x_r}.
\]
Let $f = x_{e_1} \cdots x_{e_p}$ be a squarefree monomial of degree $p$. We will show that $D_i f = 0$. If $i \notin \{e_1, \dots, e_p\}$, then the reflections $s_{i,j}^k$ fix $f$ whenever $j \notin \{e_1, \dots, e_p\}$. The contribution from $s_{i,e_j}^k$ is $-\xi^k (f/x_{e_j})$, so summing over all $k$ gives $0$. Now consider the case $i \in \{e_1, \dots, e_p\}$ (and we may as well assume $i = e_1$). Reflections of the form $s_{e_1,e_j}^k$ fix $f$ so those terms from the Dunkl operator do not contribute. Reflections of the form $s_{e_1,r}^k$ for $r \notin \{e_1, \dots, e_p\}$ produce $x_{e_2}\cdots x_{e_p}$. There are $n - p$ such $r$, so the sum of these terms is $m(n-p)x_{e_2}\cdots x_{e_i}$, which is $0$ since $n \equiv 0 \pmod p$. Similar reasoning shows that the differences of the $m$th powers of the $x_i$ are also killed by the Dunkl operators. Hence $J' \subseteq J$.

The highest degree existing in $A/J'$ is $(p-1)m$. A basis for this top degree is one element: $x_n^{(p-1)m}$. Since $D_n(x_n^s) = -c(n-1)mx_n^{s-1} = cmx_n^{s-1}$, we know that $\beta(x_n^{(p-1)m}, x_n^{(p-1)m}) = (cm)^{(p-1)m}$. Any nonzero monomial can be expressed (non-uniquely), modulo $J'$, as $f = x_{e_1}^{d_1}\cdots x_{e_{p-1}}^{d_{p-1}}$ where $1 \leq e_1 < \cdots < e_{p-1} \leq n$ and $0 \leq d_i < m$ for all $i$. Multiplying this monomial by $x_{e_1}^{m-d_1} \cdots x_{e_{p-1}}^{m-d_{p-1}}$ gives $x_{e_1}^m \cdots x_{e_{p-1}}^m \equiv x_n^{(p-1)m} \pmod J$, so the socle of $A/J'$ is in top degree. We conclude that $J=J'$ by Lemma~\ref{lem:socle}.
\end{proof}

\subsection{$G(m,m,n)$ for $m>1$, $p$ does not divide $n$} \label{sec:Gmmn-pnotdividen}

Let $G = G(m,m,n)$ in characteristic $p$ where $p$ does not divide $n$ or $m$. Write $n \equiv i \pmod p$ where $0 < i < p$. Let $A = K[x_1, \dots, x_n]$. Set $e_1(x), \dots, e_n(x)$ to be the elementary symmetric functions in $x_1, \dots, x_n$. Let $J' \subset A$ be the ideal generated by $e_1(x^m), \dots, e_n(x^m)$ and all squarefree monomials of degree $i$.

\begin{lemma}
$J' \subseteq J$.
\end{lemma}

\begin{proof}
One shows that the squarefree monomials of degree $i$ are killed by the Dunkl operators in exactly the same way as in the proof of Proposition~\ref{prop:Gmmn-pdivides}. The action of the Dunkl operators on the elementary symmetric functions in $x_1^m, \dots, x_n^m$ must also be 0 because they are invariants of $G$. 
\end{proof}

For the rest of this section, we do not make any assumptions on the characteristic of $K$ unless otherwise stated.

Let $T_i$ be the ideal in $A$ generated by all squarefree monomials of degree $i$. Then $A/T_i$ is a Cohen--Macaulay algebra of Krull dimension $i-1$. In fact, this algebra has a linear resolution \cite[Theorem 3]{eagonreiner}, and hence is a level algebra. The zero locus of $T_i$ is the set of points $(x_1, \dots, x_n)$ such that at least $n-i+1$ coordinates are equal to 0. This is a union of $\binom{n}{i-1}$ linear spaces of dimension $i-1$, so the degree of this variety is $\binom{n}{i-1}$. Note that $T_i$ is a radical ideal. 

\begin{lemma}
The Hilbert series of $A/T_i$ is
\[
(\sum_{j=0}^{i-1} \binom{n-i+j}{n-i} t^j) \left/ (1-t)^{i-1} \right.
\]
\end{lemma}

\begin{proof}
Write the Hilbert series of $A/T_i$ as $H(t) / (1-t)^{i-1}$. Since $A/T_i$ is Cohen--Macaulay, the degree of $H$ is the regularity of $A/T_i$ \cite[Corollary 4.8]{syzygies} which is $i-1$ since it has a linear resolution, and $H(1) = \deg(A/T_i) =  \binom{n}{i-1}$. Furthermore, since the Hilbert functions of $A$ and $A/T_i$ agree in degrees up to $i-1$, we conclude that $H(t) = 1 + h_1 t + \cdots + h_{i-1} t^{i-1}$ where $h_j = \binom{n-i+j}{n-i}$ is the dimension of the space of degree $j$ polynomials in $n-(i-1)$ variables.
\end{proof}

\begin{proposition} \label{prop:socledim}
The Hilbert series of $A/J'$ is
\[
(\sum_{j=0}^{i-1} \binom{n-i+j}{n-i} t^j) \cdot \prod_{j=1}^{i-1} \frac{ 1 - t^{jm} }{ 1 - t}. 
\]
The socle of $A/J'$ is concentrated in top degree and has dimension $\binom{n-1}{i-1}$.
\end{proposition}

\begin{proof}
Since $e_1, \dots, e_n$ form a homogeneous system of parameters on $A$ and $e_i, e_{i+1}, \dots, e_n \in T_i$, we conclude that $e_1, \dots, e_{i-1}$ is a homogeneous system of parameters for $A/T_i$ since $\dim A/T_i = i-1$. Therefore the same is true for $e_1(x^m), \dots, e_{i-1}(x^m)$. We have already seen that $A/T_i$ is Cohen--Macaulay, so in fact $e_1(x^m), \dots, e_{i-1}(x^m)$ is a regular sequence on $A/T_i$, and the result follows. The statement that the socle of $A/J'$ is concentrated in top degree follows from the fact that $A/T_i$ is a level algebra. The other statements follow from the above discussion.
\end{proof}

\begin{proposition}
The top degree of $A/J'$ is isomorphic to the Specht module $S_{(n-i+1,i-1)}$, which is equivalent to $\bigwedge^{i-1}\fh$ where $\fh$ is the $n-1$ dimensional reflection representation of $\Sigma_n$ and $G(m,m,n)$ acts through the surjection $G(m,m,n) \to \Sigma_n$ that sends a generalized permutation matrix to its underlying permutation.
\end{proposition}
  
\begin{proof}
We will show that this is true over $\bZ$. To check that it is an exterior power, we can work over $\bQ$, so we assume $K = \bQ$. First suppose that $m=1$. For any partition $\lambda$, the function $n(\lambda) = \sum (j-1)\lambda_j$ denotes the lowest degree of $\bQ[x_1,\ldots,x_n]$ in which $S_{\lambda}$ appears. We see that $n((n-i+1,i-1)) = \binom{i}{2}$. For general $m$, we see that $G(m,m,n)$ only acts via its quotient $\Sigma_n$ on polynomials of degrees divisible by $m$, so the lowest degree that $S_{(n-i+1,i-1)}$ appears in is $m \binom{i}{2}$.

Let $P$ be the quotient of $\bQ[x_1,\ldots,x_n]$ by the squarefree monomials of degree $i$. We see that the $S_{(n-i+1,i-1)}$ in degree $m\binom{i}{2} $ of $\bQ[x_1,\ldots,x_n]$ has as its basis the Garnir polynomials $(x^m_{t_1} - x^m_{t_2})(x^m_{t_1} - x^m_{t_3}) \cdots (x^m_{t_{i-1}} - x^m_{t_i})$ where $(t_1,\ldots,t_i)$ is the first column of a standard filling for $(n-i+1,i-1)$. When multiplying this out, there are terms with less than $i$ indices, so it is not killed by the squarefree degree $i$ monomials. Therefore, $\ch P_{m\binom{i}{2}} = \chi_{(n-i+1,i-1)} + \cdots$. Let $Q$ be the quotient of $P$ by the $e_j(x^m)$. We then see that 
\[
\ch P_{m\binom{i}{2}} = \sum_{\alpha = 1}^{i-1} \langle e_\alpha(x^m) \rangle \otimes P_{m\binom{i}{2} - m\alpha} + \ch Q_{m\binom{i}{2}}.
\]
The image of $\langle e_\alpha(x^m) \rangle$ is equivalent to the trivial representation. We have already discussed that $S_{(n-i+1,i-1)}$ does not appear in $P$ in degrees strictly smaller than $m\binom{i}{2}$, so none of these irreducibles in the sum can be $S_{(n-i+1,i-1)}$. Therefore, we have that $Q_{m\binom{i}{2}}$ contains the Specht module $S_{(n-i+1,i-1)}$. By a dimension count (Proposition~\ref{prop:socledim}) they must be equal. Since this is true over $\bQ$, it is true over $\bZ$ and we can simply reduce modulo $p$. 
\end{proof}

If $p$ does not divide $n$, then $\bigwedge^{i-1} \fh$ is an irreducible representation of the symmetric group, and hence of $G(m,m,n)$. By Lemma~\ref{lem:socle}, to show that $A/J'$ is an irreducible $\bH(G)$-module, it remains to show that $\beta$ is nonzero on the top degree of $A/J'$. But we have been unable to show this.

\begin{example}
For $i=2$ (and $n \equiv 2 \pmod p$), we have $H(t) = 1 + (n-1)t$. So the Hilbert series of $A/J'$ is 
\[
(1+(n-1)t)(1 + t + \cdots + t^{m-1}) = 1 + nt + nt^2 + \cdots + nt^{m-1} + (n-1)t^m . 
\]
The degree $m$ part of $A/J'$ is spanned by $x_1^m, \dots, x_n^m$ modulo $x_1^m + \cdots + x_n^m$. Since $D_1(x_1^s) = -cm(n-1)x_1^{s-1}$ and $n \equiv 2 \pmod p$, we get that $\beta(x_1^m,x_1^m) = (-cm)^m$. Thus $\beta$ is nonzero on the degree $m$ part of $A/J'$, so $A/J'$ is irreducible.
\end{example}

\section{Dihedral groups} \label{sec:dihedral}

In this section, we focus on the groups $G(m,m,2)$, which are the symmetry groups of regular $m$-gons, i.e., dihedral groups of order $2m$. We now use the notation $x=x_1$ and $y=x_2$. Let $\xi$ be a primitive $m$th root of unity. As usual, we assume that $p$ does not divide $m$ and that $p \ne 2$, so $p$ does not divide the order of the group.

When $m$ is even, $G(m,m,2)$ has $\frac{m}{2} + 3$ conjugacy classes. We index the representatives as $r_i$, with $r_0$ being the identity, $r_{-2}$ as $\left( \begin{smallmatrix} 0&1 \\ 1&0\end{smallmatrix} \right)$, $r_{-1}$ as $\left( \begin{smallmatrix} 0&\xi \\ \xi^{-1} & 0 \end{smallmatrix} \right)$, and $r_j$ for $1 \le j \le \frac{m}{2}$ as $\left( \begin{smallmatrix} \xi^j&0 \\ 0 & \xi^{-j} \end{smallmatrix} \right)$. Then $r_{-1}$ and $r_{-2}$ are the conjugacy classes of reflections and we set $c = c_{r_{-1}}$ and $d = c_{r_{-2}}$.

When $m$ is odd, $G(m,m,2)$ has $\frac{m-1}{2} + 2$ conjugacy classes.  We let $\xi$ be an $m$th root of unity, and we index the representatives as $r_i$, with $r_0$ being the identity, $r_{-1}$ as $\left( \begin{smallmatrix} 0&1 \\ 1&0\end{smallmatrix} \right)$, and $r_j$ for $1 \le j \le \frac{m-1}{2}$ as $\left( \begin{smallmatrix} \xi^j&0 \\ 0 & \xi^{-j} \end{smallmatrix} \right)$. Then $r_{-1}$ is the unique conjugacy class of reflections, and we set $c = d = c_{r_{-1}}$ in this case (we use both $c$ and $d$ to avoid writing separate formulas depending on the parity of $m$ below).

\subsection{Representations of dihedral groups}

The representations of $G(m,m,2)$ we describe here are indexed as $\rho_i$ for $-1 \leq i < m/2$, as well as $\rho_{-2}$ and $\rho_{-3}$ when $m$ is even. $\rho_0$ is the trivial representation. For $m$ even, $\rho_{-3}$ is the sign representation, and for $m$ odd, $\rho_{-1}$ is the sign representation. $\rho_{-1}$ and $\rho_{-2}$ are two other $1$-dimensional representations that appear when $m$ is even, but their description will not be relevant.

For $i\ge 1$, $\rho_i$ is the $2$-dimensional representation where roots of unity act by their $i$th power. We refer to the basis vectors as $e_1$ and $e_2$ for the two-dimensional representations. If an element of $G(m,m,2)$ does $x \mapsto \xi^{\ell}y$ and $y \mapsto \xi^{-\ell}x$, then it does $e_1 \mapsto \xi^{i\ell}e_2$ and $e_2 \mapsto \xi^{-i\ell}e_1$.

The four cases $\rho_i$ where $-3 \leq i \leq 0$ have the same behavior, since they are all $1$-dimensional \cite[Remark 3.31]{etingofma}, so we will just explain the case $i=0$.

\begin{proposition}
For $\tau = \rho_0$, the ideal $J$ is generated by $xy$ and $x^m + y^m$ and the Hilbert series of $A/J$ is $(1+t)(1+t+ \cdots + t^{m-1})$.
\end{proposition}

\begin{proof}
Since both of the polynomials listed are invariants of the dihedral group, they are annihilated by the Dunkl operators. Also, these two polynomials form a regular sequence, so the Hilbert series of the quotient is $(1+t)(1 + t + \cdots + t^{m-1})$. In particular, the socle of the quotient ring is concentrated in its top degree $m$ and it is spanned by $x^m$. Also, we have $D_x(x^s) = -\frac{m}{2}(c+d)x^{s-1}$ for all $s \leq m$, so $\beta(x^m,x^m) = (-\frac{m}{2}(c+d))^m \ne 0$. By Lemma~\ref{lem:socle}, we are done.
\end{proof}

\begin{proposition} \label{prop:dihedral-rho1}
For $\tau = \rho_1$, the submodule $J$ is generated by $x \otimes e_1$, $y \otimes e_2$, $x^3 \otimes e_2$, $y^3 \otimes e_1$ and the Hilbert series of $(A \otimes \rho_1) / J$ is $2 + 2t + 2t^2$.
\end{proposition}

\begin{proof}
Showing that the Dunkl operators annihilate the first two generators is trivial. Also, 
\begin{align*}
D_x(x^3 \otimes e_2) = -\frac{m}{2}(c+d)xy \otimes e_1, \qquad D_y(x^3 \otimes e_2) = \frac{m}{2}(c+d)x^2 \otimes e_1,\\
D_y(y^3 \otimes e_1) = -\frac{m}{2}(c+d)xy \otimes e_2, \qquad D_x(y^3 \otimes e_1) = \frac{m}{2}(c+d)y^2 \otimes e_2,
\end{align*}
hence the submodule generated by the $4$ listed elements is closed under applying Dunkl operators. The Hilbert series of the quotient by this submodule is $2 + 2t + 2t^2$. Its socle is concentrated in top degree and is spanned by $x^2 \otimes e_2$ and $y^2 \otimes e_1$. This is isomorphic to $\rho_1$ as a representation of the dihedral group, sending $x^2 \otimes e_2$ to $e_1$ and $y^2 \otimes e_1$ to $e_2$, so is irreducible. Since $\beta(x^2 \otimes e_2, x^2 \otimes e_2) = -(\frac{m}{2})^2(c+d)^2 \ne 0$, our quotient module is an irreducible $\bH(G)$-module by Lemma~\ref{lem:socle}.
\end{proof}

\begin{proposition}
When $m > 4$ is even and $\tau = \rho_{\frac{m}{2}-1}$, the submodule $J$ is generated by $x \otimes e_1, y \otimes e_2, x^3 \otimes e_2, y^3 \otimes e_1$ and the Hilbert series of $(A \otimes \rho_{\frac{m}{2}-1})/J$ is $2+2t+2t^2$.
\end{proposition}

\begin{proof}
The proof that $J$ is closed under applying Dunkl operators is the same as the proof of Proposition~\ref{prop:dihedral-rho1}. The top degree of the quotient module is spanned by $x^2 \otimes e_2$ and $y^2 \otimes e_1$, which is irreducible as a representation of the dihedral group. One checks that $\beta(x^2 \otimes e_2, x^2 \otimes e_2) = -(\frac{m}{2})^2(c-d)^2 \ne 0$, so our quotient is irreducible as an $\bH(G)$-module by Lemma~\ref{lem:socle}.
\end{proof}

\begin{proposition}
Set $\tau = \rho_i$ where $i > 1$. Assume that either $m$ is odd or that $m$ is even and $i < \frac{m}{2}-1$. Then the submodule $J$ is generated by $x \otimes e_1$, $x \otimes e_2$, $y \otimes e_1$, $y \otimes e_2$. In particular, $(A \otimes \rho_i) / J = \rho_i$.
\end{proposition}

\begin{proof}
Direct calculation shows that these generators are annihilated by the Dunkl operators.
\end{proof}

\subsection{Free resolutions}

We now consider the minimal free resolutions (over the polynomial ring $A$, not over the Cherednik algebra $\bH(G)$) of the $L_c(\tau)$ for the dihedral group. In all cases, the resolution has length $2$ by \eqref{eqn:AB}, so we can calculate the last term in the resolution just from the presentation of $L_c(\tau)$ (since we know the Hilbert series of $L_c(\tau)$).

\begin{compactitem}[$\bullet$]
\item For $\tau = \rho_i$ with $i \le 0$, the ideal $J$ is generated by a regular sequence of degrees $2$ and $m$, so the free resolution takes the form:
\begin{align*}
0 \leftarrow L_c(\rho_i) \leftarrow \rho_i \otimes A \leftarrow \begin{array}{c} \rho_i \otimes A(-2) \oplus\\ \rho_i \otimes A(-m) \end{array} \leftarrow \rho_i \otimes A(-m-2) \leftarrow 0.
\end{align*}

\item For $\tau = \rho_1$ and $m > 4$, the free resolution takes the form: 
\begin{align*}
0 \leftarrow L_c(\rho_1) \leftarrow \rho_1 \otimes A \leftarrow \begin{array}{c} \rho_2 \otimes A(-1) \oplus\\ \rho_2 \otimes A(-3) \end{array} \leftarrow \rho_1 \otimes A(-4) \leftarrow 0.
\end{align*}

\item For $\tau = \rho_i$ and $1 < i < \frac{m}{2}-1$, the free resolution is: 
\begin{align*}
0 \leftarrow L_c(\rho_i) \leftarrow \rho_i \otimes A \leftarrow \rho_i \otimes \fh^{*} \otimes A(-1) 
\leftarrow \rho_i \otimes \wedge^2 \fh^{*} \otimes A(-2) \leftarrow 0.
\end{align*}

However, $\fh^{*}$ is equivalent to $\rho_1$ and $\wedge^2 \fh^{*}$ is equivalent to the sign representation. We see that $\rho_i \otimes \wedge^2 \fh^{*}$ is the same as $\rho_i$ and that $\rho_i \otimes \fh^{*} \cong \rho_{i-1} \oplus \rho_{i+1}$. Then the free resolution is actually:
\begin{align*}
0 \leftarrow L_c(\rho_i) \leftarrow \rho_i \otimes A \leftarrow (\rho_{i-1} \oplus \rho_{i+1}) \otimes A(-1) 
\leftarrow \rho_i  \otimes A(-2) \leftarrow 0.
\end{align*}

\item In the case where $m$ is odd, $m \neq 3$ and $i = \frac{m-1}{2}$, the free resolution initially appears the same:
\begin{align*}
0 \leftarrow L_c(\rho_i) \leftarrow \rho_i \otimes A \leftarrow \rho_i \otimes \fh^{*} \otimes A(-1) 
\leftarrow \rho_i \otimes \wedge^2 \fh^{*} \otimes A(-2) \leftarrow 0.
\end{align*}

However, in this case $\rho_i \otimes \fh^{*}$ decomposes as $\rho_i \oplus \rho_{i-1}$ instead, so the free resolution is:
\begin{align*}
0 \leftarrow L_c(\rho_i) \leftarrow \rho_i \otimes A \leftarrow (\rho_{i-1} \oplus \rho_{i}) \otimes A(-1) 
\leftarrow \rho_i  \otimes A(-2) \leftarrow 0.
\end{align*}

\item The last general case is when $m>8$ and $\tau = \rho_{\frac{m}{2}-1}$. The free resolution is:
\begin{align*}
0 \leftarrow L_c(\rho_i) \leftarrow \rho_i \otimes A \leftarrow \begin{array}{c} (\rho_{-2} \oplus \rho_{-1}) \otimes A(-1) \oplus\\ \rho_{\frac{m}{2}-4} \otimes A(-3) \end{array} \leftarrow \rho_{\frac{m}{2}-3} \otimes A(-4) \leftarrow 0.
\end{align*}
\end{compactitem}

\subsection{Transition matrices}

Now that we have the free resolutions for all but finitely many exceptional cases (to be handled in \S\ref{sec:dihedral-exceptions}), we can consider the transition matrices from simple objects to Verma modules. We use the variable $t$ to represent grading shifts, with columns labeling simple objects $L_c(\tau)$ and rows labeling Verma modules $M_c(\tau)$. For each transition matrix, let $a_{i,j}$ represent the entry in the $i$th row and the $j$th column. 

For $m>8$ even, we have a $(\frac{m}{2}+3) \times (\frac{m}{2}+3)$ transition matrix with $\tau$ going from $\rho_{-3}$ to $\rho_{\frac{m}{2}-1}$ from left to right and top to bottom. Its nonzero entries are:

\begin{compactitem}[$\bullet$]
\item $a_{i,i} = (1-t^2)(1-t^m)$ for $1 \le i \le 4$

\item $a_{5,5} = 1+t^4$, $a_{5,6} = -t-t^3$

\item $a_{j,j} = 1+t^2$, $a_{j-1,j} = a_{j+1,j} = -t$ for $6 \le j \le \frac{m}{2}+2$

\item $a_{2,\frac{m}{2}+3} = a_{3,\frac{m}{2}+3} = -t$, $a_{\frac{m}{2},\frac{m}{2}+3} = -t^3$, $a_{\frac{m}{2}+1,\frac{m}{2}+3} = t^4$, $a_{\frac{m}{2}+3,\frac{m}{2}+3} = 1$
\end{compactitem}

For $m > 3$ odd, we have a $(\frac{m-1}{2}+2) \times (\frac{m-1}{2}+2)$ transition matrix with $\tau$ going from $\rho_{-1}$ to $\rho_{\frac{m-1}{2}}$ from left to right and top to bottom. Its nonzero entries are:

\begin{compactitem}[$\bullet$]
\item $a_{i,i} = (1-t^2)(1-t^m)$ for $i=1,2$

\item $a_{3,3} = 1+t^4$, $a_{3,4} = -t-t^3$

\item $a_{j,j} = 1+t^2$, $a_{j-1,j} = a_{j+1,j} = -t$ for $4 \le j \le \frac{m-1}{2}+1$

\item $a_{\frac{m-1}{2}+1,\frac{m-1}{2}+2} = -t$, $a_{\frac{m-1}{2}+2,\frac{m-1}{2}+2} = 1-t+t^2$
\end{compactitem}

\subsection{Exceptional cases} \label{sec:dihedral-exceptions}

There are a few exceptional cases left, so we just list the answers in these cases.

\begin{compactitem}[$\bullet$]
\item When $m=3$ and $\tau=\rho_1$, $J$ is generated by $x\otimes e_1$, $y\otimes e_2$, $x^3\otimes e_2$, $y^3\otimes e_1$. The free resolution is: 
\begin{align*}
0 \leftarrow L_c(\rho_1) \leftarrow \rho_1 \otimes A \leftarrow \rho_1 \otimes A(-1) \oplus \rho_1 \otimes A(-3) \leftarrow \rho_1 \otimes A(-4) \leftarrow 0
\end{align*}

\item When $m=4$ and $\tau = \rho_1$, $J$ is generated by $xy \otimes e_1$, $xy \otimes e_2$, $(\frac{c+d}{c-d}x^2 + y^2) \otimes e_1$, $(\frac{c-d}{c+d}x^2 + y^2) \otimes e_2$. All of these are sent to $0$ by the Dunkl operators. The Hilbert series of $L_c(\rho_1)$ is $2 + 4t + 2t^2$. The top degree of $L_c(\rho_1)$ is spanned by $x^2 \otimes e_1$, $x^2 \otimes e_2$, which is irreducible as a representation of the dihedral group, and $\beta(x^2 \otimes e_1, x^2 \otimes e_1) = -4(c-d)^2 \ne 0$. The free resolution is:
\begin{align*}
0 \leftarrow L_c(\rho_1) \leftarrow \rho_1 \otimes A \leftarrow (\rho_1 \oplus \rho_1) \otimes A(-2) \leftarrow \rho_1 \oplus A(-4) \leftarrow 0
\end{align*}

\item The final exceptional case is $\tau=\rho_{\frac{m}{2}-1}$ for $m=6,8$. The free resolution is the same as the general $\tau=\rho_{\frac{m}{2}-1}$ case, with suitable indexing modifications. The free resolutions for the cases $m=6$ and $m=8$, respectively, are:
\begin{align*}
0 \leftarrow L_c(\rho_2) \leftarrow \rho_2 \otimes A \leftarrow \begin{array}{c} (\rho_{-2} \oplus \rho_{-1}) \otimes A(-1) \oplus\\ \rho_1 \otimes A(-3) \end{array} \leftarrow (\rho_0 \oplus \rho_{-3}) \otimes A(-4) \leftarrow 0
\end{align*}
\begin{align*}
0 \leftarrow L_c(\rho_3) \leftarrow \rho_3 \otimes A \leftarrow \begin{array}{c} (\rho_{-2} \oplus \rho_{-1}) \otimes A(-1) \oplus\\ (\rho_0 \oplus \rho_{-3}) \otimes A(-3) \end{array}
\leftarrow \rho_1 \otimes A(-4) \leftarrow 0
\end{align*}
\end{compactitem}

We can now consider the transition matrices of these exceptional cases. The transition matrices not covered above are those for $m=2,3,4,6,8$, and they are all shown below:

$m=2$: Transition matrix is $(1-t^2)^2I_4$.

$m=3$: 

\noindent $\begin{pmatrix}
(1-t^2)(1-t^3)&&\\
&(1-t^2)(1-t^3)&\\
&&(1-t)(1-t^3)\\
\end{pmatrix}$

$m=4$:

\noindent $\begin{pmatrix}
(1-t^2)(1-t^4)&&&&\\
&(1-t^2)(1-t^4)&&&\\
&&(1-t^2)(1-t^4)&&\\
&&&(1-t^2)(1-t^4)&\\
&&&&(1-t^2)^2\\
\end{pmatrix}$

$m=6$:

\noindent $\begin{pmatrix}
(1-t^2)(1-t^6)&&&&&t^4\\
&(1-t^2)(1-t^6)&&&&-t\\
&&(1-t^2)(1-t^6)&&&-t\\
&&&(1-t^2)(1-t^6)&&t^4\\
&&&&1+t^4&-t^3\\
&&&&-t-t^3&1\\
\end{pmatrix}$

$m=8$:

\noindent $\begin{pmatrix}
(1-t^2)(1-t^8)&&&&&&-t^3\\
&(1-t^2)(1-t^8)&&&&&-t\\
&&(1-t^2)(1-t^8)&&&&-t\\
&&&(1-t^2)(1-t^8)&&&-t^3\\
&&&&1+t^4&-t&\\
&&&&-t-t^3&1+t^2&t^4\\
&&&&&-t&1\\
\end{pmatrix}$

\section{The rank $3$ groups $G(m,m,3)$} \label{sec:rank3}

In this section, we give some partial results on the case $G = G(m,m,3)$. We will assume that $p$ does not divide $m$ and $p \ne 3$, so $p$ does not divide the order of the group.

We deal only with the case that $3$ does not divide $m$. The representations of $G(m,m,3)$ are indexed by multi-partitions of size $3$ and length $m$, with the added relation that two multi-partitions that differ by a cyclic shift give the same representation: for example, $([2],[1], \emptyset, \emptyset)$ and $(\emptyset ,[2],[1], \emptyset)$ correspond to the same representation of $G(4,4,3)$. 

$G(m,m,3)$ has two $1$-dimensional representations: the trivial representation and the sign representation. In this case, it has one $2$-dimensional representation, which we refer to as $\gamma_0$. It corresponds to the multi-partition $([2,1],\dots)$. The $3$-dimensional representations are of the form $([2],\dots,[1],\dots)$, which we refer to as $\gamma_i$ where the $[1]$ is in the $i$th place (we number the places starting at $0$). The $6$-dimensional representations are of the form $([1],\dots,[1],\dots,[1],\dots)$, which we refer to as $\gamma_{i,j}$, where the $[1]$s are in the $0$th, $i$th, and $j$th places. We remark that in the case when $3$ divides $m$, one of these $6$-dimensional representations splits into three $2$-dimensional representations.

We have already described the character of $L_c(\tau)$ when $\tau$ is $1$-dimensional in \S\ref{sec:Gmmn-pdividesn} and \S\ref{sec:Gmmn-pnotdividen}. Let $\tau=\gamma_0$; if we consider three vectors $a_1,a_2,a_3$ that are permuted by the symmetric group in the obvious way, then the basis vectors of $\gamma_0$ are $e_1=a_1-a_3$ and $e_2=a_3-a_2$. Roots of unity have no effect on these basis vectors. The generators of $J$ take the form of a ``matrix regular sequence'', i.e., we can find generators and a way to group them into the columns of square matrices so that their determinants form a regular sequence. If $x,y,z$ are the basis vectors of $\fh^*$, we can write the six generators of $J$ as the columns of the $2 \times 2$ matrices:
\begin{align} \label{eqn:matrixreg}
\begin{pmatrix}
xyz&0\\
0&xyz\\
\end{pmatrix} \qquad
\begin{pmatrix}
x^m+y^m+z^m&0\\
0&x^m+y^m+z^m\\
\end{pmatrix} \qquad
\begin{pmatrix}
-x^m&y^m\\
z^m&-x^m\
\end{pmatrix}
\end{align}
The first two matrices are composed of invariants of $G$, so their columns are killed by the Dunkl operators. The columns of the third matrix are also easily shown to be killed. The zero locus of the determinants of the matrices is $x=y=z=0$, so they form a regular sequence. Furthermore, these matrices commute, so we can plug them into a Koszul complex of length $3$ to get a minimal free resolution of the quotient module, which tells us its Hilbert series: $2(1+t+t^2)(\frac{t^m-1}{t-1})^2$. Since $\beta(z^{2m}\otimes e_1,z^{2m}\otimes e_1)=2(mc)^{2m}$ if $m$ is even and $-(mc)^{2m}$ if $m$ is odd, $\beta$ is nonzero on the top degree; this means that $L_c(\tau)$ is irreducible by Lemma~\ref{lem:socle}.

When $\tau = \gamma_i$ for $2 \le i \le m-1$, we use the basis for $\tau$ such that permutations act as normal but roots of unity act by their $i$th power; for example, the matrix $\left(\begin{smallmatrix} 0&1&0\\1&0&0\\0&0&1\end{smallmatrix}\right)$ would permute the first two basis vectors, and the matrix $\left(\begin{smallmatrix} \xi&0&0\\0&\xi^{-1}&0\\0&0&1\end{smallmatrix}\right)$ would multiply the first basis vector by $\xi^{i}$, multiply the second basis vector by $\xi^{-i}$, and leave the third basis vector unchanged. We call the basis vectors $w_1,w_2,w_3$. Let $J' \subset A \otimes \tau$ be generated by
\[
(x,0,0),(0,y,0),(0,0,z), (yz,0,0),(0,xz,0),(0,0,xy), (z^{m-i},0,x^{m-i}),(y^{m-i},x^{m-i},0),(0,z^{m-i},y^{m-i}).
\]
These can all easily be shown to be killed by Dunkl operators. We see that the top degree of $(A \otimes \tau)/J'$, which is in degree $m-i$, is isomorphic to the sign representation times $\gamma_{m-i}$ as a representation of the group. Since $\beta(z^{m-i}\otimes w_1, z^{m-i}\otimes w_1)=(-mc)^{m-i}$, which is nonzero, this means that $(A \otimes \tau)/J'$ is irreducible as a representation of the Cherednik algebra. 

\begin{remark} \label{rmk:matrixkoszul}
The matrix regular sequence \eqref{eqn:matrixreg} is boring because all but one of the matrices is a scalar matrix. However, we have done some experiments with higher rank groups $G(m,m,n)$ (with $n \ge 4$) and we encounter more complicated matrices. The conjectural pattern is that the generators have the structure of a matrix regular sequence with $2$ of the matrices being scalar matrices. Here we give two examples for $G(2,2,4)$ on the representation $([3,1],\emptyset, \emptyset, \emptyset)$ in characteristic $7$:
\[
\begin{pmatrix}
x^2 + y^2 + z^2 + w^2&0&0\\
0&x^2 + y^2 + z^2 + w^2&0\\
 0&0&x^2 + y^2 + z^2 + w^2\\
\end{pmatrix}, \,
\begin{pmatrix}
xyzw & 0& 0\\
0&xyzw&0\\
0& 0& xyzw\\
\end{pmatrix},
\]
\[
\begin{pmatrix}
4x^2 + z^2&5x^2&3x^2 + y^2\\
5x^2 + 4y^2 - z^2&-x^2 - y^2&x^2 + 3y^2\\
2x^2&x^2 + z^2&-x^2 + y^2\\
\end{pmatrix},\,
\begin{pmatrix}
2x^4&x^4&x^4\\
 x^4 + 2x^2y^2 + z^4&x^4 + x^2y^2 + x^2z^2&3x^4 + 5x^2y^2 + y^4\\
 2x^2y^2&3x^4 + 4y^4&x^4 - x^2y^2 + 4y^4\\
\end{pmatrix}
\]
and in characteristic $11$, the scalar matrices stay the same while the other $2$ are replaced with:
\[
\begin{pmatrix}
8x^2 + z^2&9x^2&3x^2 + y^2\\
9x^2 + 8y^2 - z^2&-x^2 - y^2&x^2 + 3y^2\\
2x^2&x^2 + z^2&-x^2 + y^2\\
\end{pmatrix}, \,
\begin{pmatrix}
2x^4&x^4&3x^4\\
 x^4 + 2x^2y^2 + z^4&x^4 + x^2y^2 + x^2z^2&5x^4 + 9x^2y^2 + y^4\\
 2x^2y^2&5x^4 + 6y^4&x^4 - x^2y^2 + 6y^4\\
\end{pmatrix}
\]
We point out one serious deficiency with these presentations: the matrices do not commute and their determinants do not form a regular sequence. However, the resolution of the quotient of the free module of rank $3$ by $12$ generators given by the columns of the matrices in both cases has total Betti numbers $\rank \bF_i = 3 \binom{4}{i}$. Since our definition of matrix regular sequence is highly dependent on the choice of presentation, we might guess that there is a way to rearrange the generators so that we get $4$ commuting matrices whose determinants do form a regular sequence. 

It would be interesting to further investigate this phenomena since one can view it as a module-theoretic generalization of a complete intersection. In particular, even to find a deterministic way to find a nice presentation of the generators that would work on the examples above would be of interest. We note that a definition of matrix regular sequence is given in \cite{matrixkoszul} which guarantees that the corresponding matrix Koszul complex is exact, but it remains to be seen if one can extend the definition to a more general context.
\end{remark}

\footnotesize 
\noindent Massachusetts Institute of Technology, Cambridge, MA, USA, {\tt sheelad@mit.edu}

\noindent University of California, Berkeley, CA, USA, {\tt svs@math.berkeley.edu}

\end{document}